\documentclass[a4paper]{ifacconf}

\usepackage{natbib}            
\usepackage{graphicx}          

\usepackage{epsfig}
\usepackage{graphicx}        
\usepackage{amssymb}
\usepackage{amsmath}
\usepackage{epic}
\usepackage{eepic}
\usepackage{pstricks}
\usepackage{xspace}
\usepackage{float}
\usepackage{url}
\usepackage{auto-pst-pdf}
\usepackage{breqn}

\makeatletter
\newif\if@restonecol
\makeatother

\usepackage[ruled,vlined]{algorithm2e}

\newcommand{\bmat}{\begin{bmatrix}}
\newcommand{\emat}{\end{bmatrix}}
\newcommand{\real}{{\mathbb{R}}}

\newcommand{\edit}[1]{{#1}}
\newenvironment{proof}%
  {\slshape Proof:~~\normalfont}
  {\hfill$\lhd$\normalfont\newline}

\begin{document}

\begin{frontmatter}

\title{Efficient Desynchronization of Thermostatically Controlled Loads} 

\thanks[footnoteinfo]{The first author was supported in part by the Otto M{\o}nsted and Dir. Ib Henriksen foundations, Denmark; The second author was partially supported by AFOSR.}

\author[First]{Jan Bendtsen\thanksref{footnoteinfo}} 
\author[Second]{Srinivas Sridharan} 

\address[First]{Department of Electronic Systems, Automation and Control, Aalborg University, Denmark (e-mail: dimon@es.aau.dk).}                                              
\address[Second]{Department of Mechanical and Aerospace Engineering, University of California, San Diego, California (email: srsridharan@eng.ucsd.edu) }


\begin{abstract}
This paper considers demand side management in smart power grid systems containing significant numbers of thermostatically controlled loads such as air conditioning systems, heat pumps, etc. Recent studies have shown that the overall power consumption of such systems can be regulated up and down centrally by broadcasting small setpoint change commands without significantly impacting consumer comfort. However, sudden simultaneous setpoint changes induce undesirable power consumption oscillations due to sudden synchronization of the on/off cycles of the individual units. In this paper, we present a novel algorithm for counter-acting these unwanted oscillations, which requires neither central management of the individual units nor communication between units. We present a formal proof of convergence of homogeneous populations to desynchronized status, as well as simulations that indicate that the algorithm is able to effectively dampen power consumption oscillations for both homogeneous and heterogeneous populations of thermostatically controlled loads.  
\end{abstract}

\end{frontmatter}


\section{Introduction}
\label{sec:intro}
With growing penetration of renewable energy sources in modern power grids, \emph{demand side management} has been gaining attention as a means of achieving better balancing between supply and demand (\cite{MohsenianRad:2010}, \cite{Strbac:2008}, \cite{Short:2007}). Indeed, it appears that the higher intermittency and lack of dispatchability associated with increased dependence on renewable energy sources can be taken care of more effectively by electrical loads than by conventional generators, which are typically not designed for fast up- and down regulation (\cite{Strbac:2008}, \cite{Klobasa:2010}). 

Various technologies are currently being considered in the context of demand side management; coordinated charging of batteries, e.g., in electric vehicles (\cite{Mets:2010}), deliberate scheduling of loads with flexible deadlines (\cite{Petersen:2012}) as well as allowing local consumers with slow dynamics (large time constants) to store more or less energy at convenient times and thereby adjusting the momentary consumption (e.g., \cite{Moslehi:2010}), among others. In particular, so-called \emph{thermostatically controlled loads}  (TCLs), such as deep freezers, refrigerators, local heat pumps etc., are showing great potential in this context, since they account for a large volume of consumption in most countries with significant renewable penetration---for instance, as of 2009, about 87 percent of all US homes were equipped with air conditioning.\footnote{
\url{http://www.eia.gov/consumption/residential/reports/2009/air-conditioning.cfm}}

Thus, at least in theory, manipulating the operating conditions of large populations of units slightly while avoiding discomfort to end users appears to be attractive, given that it can be achieved simply by broadcasting setpoint changes. Control strategies based on this principle were considered in \cite{Callaway:2009}, \cite{Kundu:2011}, \cite{Perfumo:2012} and \cite{Bashash:2012}, among others. In these strategies, subtle changes in the thermostat setpoint temperature  (less than a degree Celcius) are transmitted to all the participating consumers simultaneously. Such small variations in the thermostat setpoints are expected to remain almost unnoticed by the customers, but if a sufficiently large number of units are shifted in the same direction, the overall power consumption can be shifted quickly by significant amounts, allowing to compensate for power production fluctuations in the grid. 

There is an intrinsic problem with this approach, however. When the setpoint is changed concurrently in a large number of TCLs, their on/off cycles tend to become \emph{synchronized}, which leads to large unwanted fluctuations (damped oscillations) in the overall power consumption \cite{Callaway:2009}, \cite{Kundu:2011}. If the entire population of TCLs consists of identical units (\emph{homogeneous}), it is possible to counteract the fluctuations by means of a centralized control law. In the much more realistic case of non-identical (\emph{heterogeneous}) units, on the other hand, the fluctuations are far harder to remove using centralized control strategies. To the authors' knowledge, the only solution strategy presented in the literature so far is \cite{Kundu:2012}, where a strategy in which the individual TCLs deliberately slow down their transition in order to avoid synchronization was proposed.  

In this paper, we propose a novel, decentralized algorithm for avoiding synchronization without having to actively communicate between units. The algorithm proposed here is fundamentally different from the one presented in \cite{Kundu:2012}) in that it first adjusts the temperature bands in which the TCLs operate to achieve early (randomized) desynchronization, and then makes use of a technique inspired by contention-based media access protocols to adaptively adjust the on/off cycles to achieve as little synchronization among units as possible. We present a formal proof of convergence of homogeneous populations to this desynchronized status, as well as simulations that indicate that the algorithm is able to effectively dampen power consumption oscillations for both homogeneous and heterogeneous populations of thermostatically controlled loads.

The outline of the rest of the paper is as follows. Section \ref{sec:problem} first presents the simplified TCL model we shall employ, along with the synchronization issue. Section \ref{sec:algorithm} then presents the proposed algorithm along with the aforementioned convergence result. Section \ref{sec:example} illustrates the feasibility of the approach through two simulation examples. Finally, Section \ref{sec:disc} offers some concluding remarks.


\section{Problem setup}
\label{sec:problem}
We consider a population of $N$ individual temperature controlled loads. The TCLs are modeled in a generic manner, which is deliberately kept as simple as possible. Nonetheless, as will be illustrated with a few simulation examples, the behavior of a large population of simple units can be fairly complex.

\subsection{Model}
The individual TCLs are modeled as follows. Let the internal and ambient temperatures of the volumes (living spaces, refrigerators, cold storages etc.) affected by the action of the heating/cooling hardware of the $i$'th consumer be denoted $\theta_i$ and $\theta_{\infty, i}$, respectively, and assume that the hardware is purely on/off-regulated. A simple model for the internal temperature can then be formulated as \cite{Malhame:1985}, \cite{Callaway:2009}, \cite{Bashash:2012}:
\begin{eqnarray}
\dot{\theta}_i(t) &=& \frac{1}{R_i C_i}(\theta_{\infty, i} - \theta_i(t) - s_i(t)R_i P_i),  \label{eq:tcldynamics} \\
s_i(t) &=& 
\begin{cases} 
	0        &\text{if } s_i(t^-) = 1 \wedge T_i(t) \leq T_{\text{min},i} \\
	1        &\text{if } s_i(t^-) = 0 \wedge T_i(t) \geq T_{\text{max},i} \\
	s_i(t^-) & \text{otherwise}
\end{cases}   \label{eq:tclswitch}
\end{eqnarray}
for $i = 1, 2, \ldots, N$, where $C_i$ is the thermal capacitance of the consumer, $R_i$ is the corresponding thermal resistance and $P_i$ is the (constant) heating/cooling power supplied by the hardware when switched on. $s_i \in \{0,1\}$ is a binary switching variable that determines whether or not the hardware is turned on; basically, it switches status whenever the internal temperature encounters the limits of a pre-set temperature span $[\theta_{\text{min},i}, \theta_{\text{max},i}] \subset \real$ ($t^- = t - \epsilon$, where $\epsilon$ is an infinitesimal positive value). Note that both heating and cooling systems can be modeled by (\ref{eq:tcldynamics})--(\ref{eq:tclswitch}), by choosing the sign of the term $s_i(t)R_i P_i$ appropriately.

\begin{figure}
\begin{center}
\includegraphics[width=0.8\linewidth]{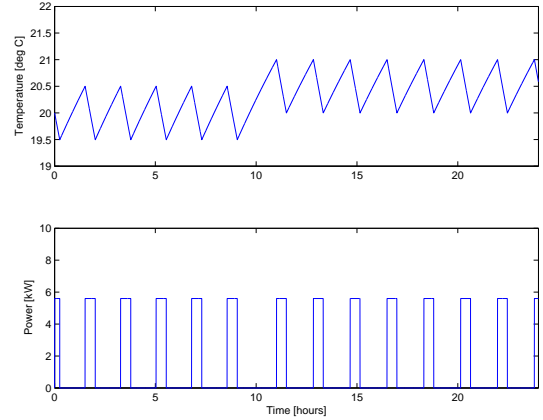}
\caption{Temperature and power consumption of a single TCL}
\label{fig:singleTCL} 
\end{center}
\end{figure}

The temperature limits $\theta_{\text{min},i}$ and $\theta_{\text{max},i}$ are related to the $i$'th consumer's setpoint $\theta_{\text{sp},i}$ through the relations
\[ \theta_{\text{min},i} = \theta_{\text{sp},i} - \frac{\Delta}{2}, \quad 
   \theta_{\text{max},i} = \theta_{\text{sp},i} + \frac{\Delta}{2}  \]
where $\Delta$ is the width of the temperature interval. Furthermore, the cumulative power consumption of the population of TCLs at any given time $t$ can be computed as
\begin{equation}
P(t) = \sum_{i=1}^N \frac{P_i s_i(t)}{\eta_i}
\end{equation}
where $\eta_i$ is the coefficient of performance for the $i$'th heating/cooling unit.

Figure \ref{fig:singleTCL} illustrates the characteristic behavior of a single TCL. The parameters are chosen as $R = 2 ^\circ$C/kW, $C = 5$ kWh$/^\circ$C, $P = 14$ kW, $\eta = 2.5$, $\theta_{\text{sp}} = 20^\circ$C, $\Delta = 1^\circ$C and $\theta_{\infty} = 28^\circ$C, respectively. The temperature setpoint is changed to $\theta_{\text{sp}} = 20.5^\circ$C at time $t=10$ h, resulting in a slightly longer period in which the cooling power is turned off, until the new operating interval is reached.

\subsection{Population behavior}
As illustrated in \cite{Callaway:2009}, \cite{Perfumo:2012}, \cite{Bashash:2012} and others, the power consumption of populations of units with the individual dynamics given above can be manipulated via broadcasts of small setpoint changes. Figure \ref{fig:homo} shows the behavior of a population of 10,000 homogeneous (identical) TCLs with the same set of parameters as in figure \ref{fig:singleTCL} (only the temperature profiles of 25 TCLs are shown in the figure). The devices are started at random temperatures  unifomly distributed within $[\theta_{\text{min}} ,  \theta_{\text{max}}]$, and their on/off cycles are thus desynchronized, yielding a fairly smooth power consumption trajectory. At time $t = 10$ h, all the TCLs are subjected to a common step in the setpoint of 0.5$^\circ$C. As can be seen, the overall power consumption drops by an amount proportional to the size of the step, but significant oscillations appear in the power consumption due to synchronization of the individual units' on/off cycles.

\begin{figure}
\begin{center}
\includegraphics[width=0.8\linewidth]{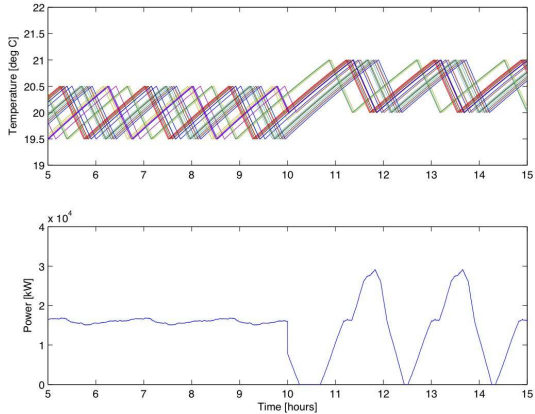}
\caption{Temperature and power consumption of a population of 10,000 homogeneous TCLs}
\label{fig:homo} 
\end{center}
\end{figure}

\begin{figure}
\begin{center}
\includegraphics[width=0.8\linewidth]{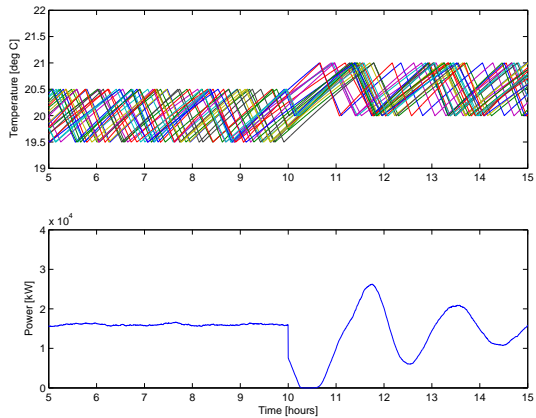}
\caption{Temperature and power consumption of a population of 10,000 heterogeneous TCLs}
\label{fig:10000TCLs} 
\end{center}
\end{figure}

Figure \ref{fig:10000TCLs} illustrates the behavior of a population of TCLs under the same simulation circumstances, but with $C_i$ chosen from a normal distribution with mean 5 and spread 0.5 kWh$/^\circ$C (again, only the temperature profiles of 25 TCLs are shown in the figure). A step of 0.5$^\circ$C in the setpoint is broadcast to all the TCLs at  time $t=10$ h. Notice how the power consumption is quite smooth before the step because the on/off cycles of the TCLs are desynchronized, whereas immediately after the step the TCLs become synchronized and the power consumption again exhibits large fluctuations, which in this case die out slowly. The oscillations die out faster if the parameters vary more (e.g., if $C_i$ is chosen from a distribution with larger spread); furthermore, the amplitude is relatively larger for larger populations of units. However, qualitatively the behavior remains the same. The oscillations are known as \emph{parasitic oscillations} and are hard to remove via centralized control signals (\cite{Kundu:2012}).

\subsection{Problem formulation}
In order to reduce the aforementioned oscillations in power consumption after the setpoint change broadcasts, it is clearly necessary to deliberately desynchronize the TCLs without interrupting their operation. However, a centralized algorithm for doing so is undesirable; it is considered infeasible to keep close track of the internal states (temperatures, set points etc.) of all the units in a centralized manner, since doing so would require regular measurement feedback from all the devices, which would give rise to a very heavy communication and computational load on the system. 

Thus, we look for a low-complexity, \emph{decentralized} algorithm  that can be implemented locally in each TCL, and which satisfies the following requirements:

\begin{enumerate}
  \item[R1] The algorithm may only use local information; this means that the TCLs may not communicate with each other, and labels identifying each individual TCL may not be pre-assigned 
  \item[R2] Communication with the power supply utility must be limited to broadcast from the utility; no communication originating from the TCLs is allowed
  \item[R3] The general operation of each individual TCL may not be altered in a manner that is detrimental to the user comfort; for example, the unit may not be deliberately kept turned off or on long enough for the temperature to leave the interval $[\theta_{\text{min}} ,  \theta_{\text{max}}]$   
\end{enumerate}

\section{Desynchronization algorithm}
\label{sec:algorithm}
In the following, we present a completely decentralized algorithm that is able to reduce the power oscillations discussed above without violating any of the requirements R1--R3. The only significant assumption we require to be satisfied is that the individual TCLs have access to time stamps of changes in the total power consumption. While this may be considered a form of global information, its usage is limited to identifying the times at which the immediate predecessor and the immediate successor of a local TCL changes state. It is not used to actually identify said units; i.e., we only require local temporal information. The information about changes in global power consumption may either be measured locally or broadcast from the utility, thus avoiding violating requirements R1 and R2.

The algorithm comprises two main components. Firstly, the algorithm is reset whenever each TCL in question receives a setpoint change $\theta_{\text{sp},i} \to \theta_{\text{sp},i} + \delta$. At this point, each TCL immediately narrows its operation interval by a random value $\alpha \in [0, \Delta/2]$, causing the dynamics of unit $i$ momentarily to be governed by the dynamics 
\begin{eqnarray*}
\dot{\theta}_i(t) &=& \frac{1}{R_i C_i}(\theta_{\infty, i} - \theta_i(t) - s_i(t)R_i P_i),  \label{eq:tcldynamics} \\
s_i(t) &=& 
\begin{cases} 
	0        &\text{if } s_i(t^-) = 1 \wedge \theta_i(t) \leq \theta_{\text{sp},i} + \delta + \frac{\Delta_i}{2}- \alpha_i \\
	1        &\text{if } s_i(t^-) = 0 \wedge \theta_i(t) \geq \theta_{\text{sp},i} + \delta - \frac{\Delta_i}{2} + \alpha_i \\
	s_i(t^-) & \text{otherwise}
\end{cases}  
\end{eqnarray*}
This operation rapidly induces a large degree of desynchronization among the population of TCLs without violating requirement R3, as illustrated in Figure \ref{fig:initialdesync}. However, randomly choosing the temperature bounds may yield very tight bounds, which in turn may give rise to constant rapid on/off-switching. This is generally not desirable, so we drive the bounds back to their original settings by introducing the dynamics
\begin{equation}
  \dot{\alpha}(t) = -a \alpha(t)
\end{equation}
for some appropriately chosen $a > 0$, as also illustrated in the figure.

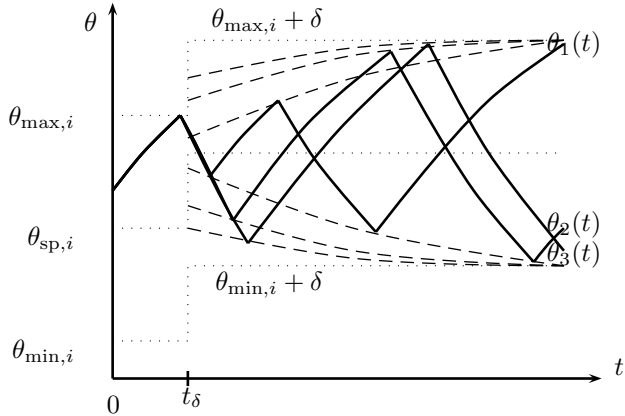
\begin{figure}
\begin{center}
\psset{unit=1mm,linewidth=1pt}
\begin{pspicture}(0,0)(80,60)
%
\psline{->}(10,5)(75,5)
\put(8,0){\makebox( 3, 3)[r]{$0$}}
\put(75,4){\makebox( 3, 5)[r]{$t$}}
\psline{-}(20,4)(20,6)
\put(19,0){\makebox( 3, 5)[r]{$t_\delta$}}
\psline{->}(10,5)(10,55)
\put(5,50){\makebox( 3, 5)[r]{$\theta$}}
\psline[linewidth=0.5pt,linestyle=dotted](10,25)(20,25)(20,35)(70,35)
\put(2,22){\makebox(3, 3)[r]{$\theta_{\text{sp},i}$}}
\psline[linewidth=0.5pt,linestyle=dotted](10,10)(20,10)(20,20)(70,20)
\put(2,7){\makebox(3, 3)[r]{$\theta_{\text{min},i}$}}
\put(2,37){\makebox( 3, 5)[r]{$\theta_{\text{max},i}$}}
\psline[linewidth=0.5pt,linestyle=dotted](10,40)(20,40)(20,50)(70,50)
\put(35,16){\makebox(3, 3)[r]{$\theta_{\text{min},i}+\delta$}}
\put(35,50){\makebox( 3, 5)[r]{$\theta_{\text{max},i}+\delta$}}
\pscurve[linewidth=0.5pt,linestyle=dashed]{-}(20,37)(42,45)(70,50)
\pscurve[linewidth=0.5pt,linestyle=dashed]{-}(20,33)(42,25)(70,20)
\pscurve[linewidth=0.5pt,linestyle=dashed]{-}(20,42)(42,48)(70,50)
\pscurve[linewidth=0.5pt,linestyle=dashed]{-}(20,28)(42,22)(70,20)
\pscurve[linewidth=0.5pt,linestyle=dashed]{-}(20,45)(42,49)(70,50)
\pscurve[linewidth=0.5pt,linestyle=dashed]{-}(20,25)(42,21)(70,20)
\pscurve[linewidth=1pt]{-}(10,30)(14,35)(19,40)
\pscurve[linewidth=1pt]{-}(19,40)(21,36.3)(23,32)
\pscurve[linewidth=1pt]{-}(23,32)(27,37)(32,42)
\pscurve[linewidth=1pt]{-}(32,42)(37,34)(45,24.5)
\pscurve[linewidth=1pt]{-}(45,24.5)(60,42)(70,49.5)
\put(72,48){\makebox( 3, 3)[r]{$\theta_1(t)$}}
\pscurve[linewidth=1pt]{-}(10,30)(14,35)(19,40)
\pscurve[linewidth=1pt]{-}(19,40)(21,36.3)(26,26)
\pscurve[linewidth=1pt]{-}(26,26)(36,39)(47,48.5)
\pscurve[linewidth=1pt]{-}(47,48.5)(58,30)(66,20.5)
\pscurve[linewidth=1pt]{-}(66,20.5)(68,23)(70,25)
\put(72,24){\makebox( 3, 3)[r]{$\theta_2(t)$}}
\pscurve[linewidth=1pt]{-}(10,30)(14,35)(19,40)
\pscurve[linewidth=1pt]{-}(19,40)(25,28)(28,23)
\pscurve[linewidth=1pt]{-}(28,23)(40,38)(52,49.5)
\pscurve[linewidth=1pt]{-}(52,49.5)(60,35)(70,22)
\put(72,20){\makebox( 3, 3)[r]{$\theta_3(t)$}}
\end{pspicture}
\end{center}
\caption{Illustration of TCL desynchronization algorithm before and after a step occurring at $t = t_\delta$. The temperature profiles $\theta_i$ of three TCLs indexed by $i = 1, 2, 3$ are shown with full lines. The dashed lines indicate the operating intervals $[\theta_{\text{sp},i} + \delta + \Delta_i/{2}- \alpha_i(t) , \theta_{\text{sp},i} + \delta - \Delta_i/{2} + \alpha_i(t)]$. Due to different random values of $\alpha_i$, the TCLs become desynchronized after the step even though they were completely synchronized before and immediately after $t_\delta$. The operation intervals converge exponentially back to their original values.}
\label{fig:initialdesync} 
\end{figure}

Now, even though most of the synchronization among the population of TCLs is likely to have disappeared with the random shrinking of the operating interval, it may still happen that some TCLs remain synchronized, or close to synchronized. The second component of the proposed algorithm is designed to deal with any `left-over' synchronization among TCLs that are close to being identical. Here, we make use of the fact that the behavior of the TCLs is periodic in steady state operation. Let the period be denoted $T$. Essentially, this part of the desynchronization algorithm repeatedly forces the devices to switch status at some point within the interval $[0,T]$, and adaptively adjusts this enforced switch timing to achieve maximal spread across the period; see Figure \ref{fig:period}.

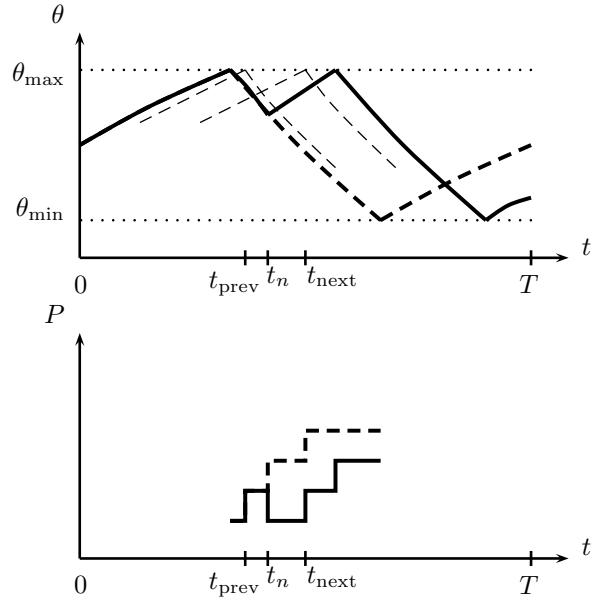
\begin{figure}
\begin{center}
\psset{unit=1mm,linewidth=1pt}
\begin{pspicture}(0,0)(80,75)
%
\psline{->}(10,45)(75,45)
\put(8,40){\makebox( 3, 3)[r]{$0$}}
\put(75,44){\makebox( 3, 5)[r]{$t$}}
\psline{-}(32,44)(32,46)
\put(31,40){\makebox( 3, 3)[r]{$t_{\text{prev}}$}}
\psline{-}(35,44)(35,46)
\put(35,40){\makebox( 3, 5)[r]{$t_n$}}
\psline{-}(40,44)(40,46)
\put(44,40){\makebox( 3, 5)[r]{$t_{\text{next}}$}}
\psline{-}(70,44)(70,46)
\put(68,40){\makebox( 3, 3)[r]{$T$}}
\psline{->}(10,45)(10,75)
\put(5,75){\makebox( 3, 5)[r]{$\theta$}}
\put(5,50){\makebox(3, 3)[r]{$\theta_{\text{min}}$}}
\put(5,67){\makebox( 3, 5)[r]{$\theta_{\text{max}}$}}
\psline[linestyle=dotted](10,50)(70,50)
\psline[linestyle=dotted](10,70)(70,70)
\pscurve[linewidth=1.5pt,linestyle=dashed]{-}(10,60)(19,65)(30,70)
\pscurve[linewidth=1.5pt,linestyle=dashed]{-}(30,70)(39,60)(50,50)
\pscurve[linewidth=1.5pt,linestyle=dashed]{-}(50,50)(59,55)(70,60)
\pscurve[linewidth=1.5pt]{-}(10,60)(19,65)(30,70)
\pscurve[linewidth=1.5pt]{-}(30,70)(32,68)(35,64)
\pscurve[linewidth=1.5pt]{-}(35,64)(41,68)(44,70)
\pscurve[linewidth=1.5pt]{-}(44,70)(53,60)(64,50)
\pscurve[linewidth=1.5pt]{-}(64,50)(67,52)(70,53)
\pscurve[linewidth=0.5pt,linestyle=dashed]{-}(18,63)(28,68)(32,70)
\pscurve[linewidth=0.5pt,linestyle=dashed]{-}(32,70)(35,66)(44,57)
\pscurve[linewidth=0.5pt,linestyle=dashed]{-}(26,63)(36,68)(40,70)
\pscurve[linewidth=0.5pt,linestyle=dashed]{-}(40,70)(43,66)(52,57)
\psline{->}(10,5)(75,5)
\put(8,0){\makebox( 3, 3)[r]{$0$}}
\put(75,4){\makebox( 3, 5)[r]{$t$}}
\psline{-}(32,4)(32,6)
\put(31,0){\makebox( 3, 3)[r]{$t_{\text{prev}}$}}
\psline{-}(35,4)(35,6)
\put(35,0){\makebox( 3, 5)[r]{$t_n$}}
\psline{-}(40,4)(40,6)
\put(44,0){\makebox( 3, 5)[r]{$t_{\text{next}}$}}
\psline{-}(70,4)(70,6)
\put(68,0){\makebox( 3, 3)[r]{$T$}}
\psline{->}(10,5)(10,35)
\put(5,35){\makebox( 3, 5)[r]{$P$}}
\psline[linewidth=1.5pt,linestyle=dashed]{-}(30,10)(32,10)(32,14)(35,14)(35,18)(40,18)(40,22)(44,22)(50,22)
\psline[linewidth=1.5pt]{-}(30,10)(32,10)(32,14)(35,14)(35,10)(40,10)(40,14)(44,14)(44,18)(50,18)
\end{pspicture}
\end{center}
\caption{Illustration of TCL desynchronization algorithm. The top figure shows the temperature profile of a TCL with (thick full line) and without (thick dashed line) enforced switching at time $t = t_n$. The thin dashed lines show parts of the temperature traces of the preceding and succeeding TCLs. Their switch times are recorded by the local TCL as $t_{\text{prev}}$ and $t_{\text{next}}$, respectively. The bottom figure shows the relevant part of the corresponding power traces.}
\label{fig:period} 
\end{figure}

In short, each TCL records the switch timings of the immediately preceding and the immediately succeeding TCLs; then it sets the new enforced switch time to the mean of the switch times of the succeeding and preceding TCLs, i.e., 
\begin{equation}
  \label{eq:enforced}
  t_{k+1} = \frac{t_{\text{prev}} + t_{\text{prev}}}{2} 
\end{equation} 
where $t_{k+1} \in [0, T]$ denotes the time of the enforced switch within the following period, modulo $T$.

The enforced switch timings are initially distributed randomly; the TCL which randomly obtains the smallest value will then not be able detect a switch prior to its own. When it recognizes this fact, it sets its own switch timing to $0$ and maintains it at that value (modulo $T$); it is thus not necessary to assign a global label to the first and/or last TCL a priori.

Let $x^j_k$ denote the time of the $j$'th forced transition, modulo $T$, at iteration $k$; that is, $0 = x^1_k < x^2_k < \dots < x^N_k < T$. We gather these switch timings in a vector $\mathbf{x}_k \in [0 , T]^{N}$.
We can then prove the following asymptotic result:

\begin{thm}
If all switch timings for $i = 1,\ldots,N$ are updated according to (\ref{eq:enforced}), the globally attracting switch timing state $\mathbf{x}^*$ (modulo $T$) is described by
\edit{
\begin{align}
\mathbf{x}^*\,\, \mathop{mod}\,\, T  = \bmat 0 & \frac{T}{N}, & \frac{2T}{N}, &  \ldots & ,\frac{T(N-1)}{N}\emat^{'}.
\end{align}
}
\end{thm}

\begin{proof}
The algorithm described above for the evolution of the state (time) dynamics requires that each state be the mean value of the 
prior and next state with the caveat that the $N$'th element should update its value based on the next pulse i.e. the value of the first state element in the next time instant. Hence, we first extend the state vector by placing the new value of the first state into the fictitious $N+1$ th dimension of  the state vector $\mathbf{x}$ to generate a new state vector,
%
denoted $\tilde{\mathbf{x}}$. The update dynamics of the enforced switch times can be written as
\begin{eqnarray}
\tilde{\mathbf{x}}_{k+1}  &=&  
  \bmat 1 & 0 & 0 & 0 & \dots & 0 & 0 & 0 \\
        \frac{1}{2} & 0 & \frac{1}{2} & 0 & \dots & 0 & 0 & 0 \\ 
        0 & \frac{1}{2} & 0 & \frac{1}{2} & \dots & 0 & 0 & 0 \\ 
        \vdots & & \ddots & \ddots & \ddots & \vdots & \vdots \\
        0 & 0 & 0 & 0 & \dots & \frac{1}{2} & 0 & \frac{1}{2} \\
        0 & 0 & 0 & 0 & \dots & 0 & 0 & 1 \emat \tilde{\mathbf{x}}_k \nonumber \\
  \label{eq:dynamics}
        &=& \Gamma \tilde{\mathbf{x}}_{k}.
\end{eqnarray}
which implies that
\begin{equation*}
  \lim_{k\to\infty} \tilde{\mathbf{x}}_{k} =   \lim_{k\to\infty} \Gamma^k \tilde{\mathbf{x}}_{0} 
\end{equation*}

Using eigenvalue decomposition, we may write the aforementioned limit as  
\begin{align}
\lim_{k \to \infty} \tilde{\mathbf{x}}_k 
= \lim_{k \rightarrow \infty}  E \Lambda^{k} E^{-1} \tilde{\mathbf{x}}_{0} \label{eqnnameGammainfexpr}
\end{align}
where $\Lambda\in\real^{(N+1) \times (N+1)}$ is a diagonal matrix of eigenvalues of $\Gamma$ and $E\in\real^{(N+1)\times (N+1)}$ contains the corresponding eigenvectors.  
$\Gamma$ has two repeated unit eigenvalues with linearly independent  eigenvectors $\gamma$ and $\mathbf{1}-\gamma$, where $\mathbf{1} \in\real^{N+1}$ is a vector of one-elements and 
%
%
\begin{align}
  \gamma = \bmat 1 & 1-\frac{1}{N} & 1-\frac{2}{N} & \dots & 0 \emat^T
\end{align}
The remaining eigenvalues are of absolute magnitude less than $1$, implying that all except two of the diagonal elements in the eigenvalue matrix $ \Lambda^{k}$ tend to zero as $k\to\infty$. Hence, in order to obtain the limiting value of the state dynamics, it remains only to utilize the eigenvectors corresponding to the unit eigenvalues to obtain the limit in \eqref{eqnnameGammainfexpr}.  

Partition the eigenvector matrices as follows
\begin{equation*}
E = \bmat e_1 & e_2 & \cdots & e_{N+1}\emat, \quad
E^{-1} = \bmat \tilde{e}_1 & \tilde{e}_2 & \cdots & \tilde{e}_{N+1} \emat^T    
\end{equation*}
and let $e_1$ and $e_{N+1}$ denote the eigenvectors corresponding to the unit eigenvalues, i.e., $e_1 = \gamma$, $e_{N+1} = \mathbf{1}-\gamma$.
It is now possible to verify from the structure of $\Gamma$ and the partitioning of $E$ and $E^{-1}$ that  
\begin{align}
\tilde{e}_1^T =  \left[\begin{array}{cccc}1 & 0 & \cdots & 0\end{array}\right],\quad
\tilde{e}_{N+1}^T =   \left[\begin{array}{cccc}0 & 0 & \cdots & 1\end{array}\right],
\end{align}
%

Finally, substituting $e_1 = \gamma$, $e_{N+1} = \mathbf{1}-\gamma$ into \eqref{eqnnameGammainfexpr} we get
\begin{align}
\lim_{k \rightarrow \infty} \Gamma^k = \bmat \gamma & 0 & \cdots & 0 & \mathbf{1}-\gamma \emat.
\end{align}
Thus it is seen that the asymptotic limit of the state is dependent only on the first and last elements of the  state vector $\tilde{x}$. As these are fixed at $0$ and $T$ respectively, the  statement of the theorem follows as desired.
\end{proof}
It is of interest to note that if $\Gamma$ had been a positive (or irreducible) matrix, the  Perron-Frobenuis theorem would have been an elegant approach to obtain the steady state matrix directly. However, in this case, owing to the zero entires and the lack of strong connectivity of the graph associated with $\Gamma$ (a necessary and sufficient condition for irreducibility), the above more involved proof is required.

Finally, we state the full desynchronization algorithm as Algorithm 1, where we use the notation $\mathop{rand} \in [0,T]$ to denote the operation of picking a random value from a uniform distribution over the interval $[0,T]$.

\begin{algorithm}[h!]
  \KwData{Broadcast signal indicating a setpoint change $\delta$; Initial operation setpoint $\theta_{\text{sp}}$ and interval width $\Delta$; initial on/off status $s\in\{0,1\}$; on/off period time $T$; constant $a > 0$}   
  \KwResult{$t_k$ yielding maximal desynchronization for $k \to\infty$}
  \Begin{
    $\theta_{\text{sp}} \longleftarrow \theta_{\text{sp}} + \delta$ \;
    $t_0 \longleftarrow \mathop{rand} \in [0,T]$ \;
    $\alpha \longleftarrow \mathop{rand} \in [0,\Delta/2]$ \;
    $k \longleftarrow 0$\;
    \While{$k < \infty$}{
      \For{$0 < t < T$}{
        Measure $\theta(t)$ \;  
        \If{$t < t_{k}$}{
          $t_{\text{prev}} \longleftarrow$ time of last transition before $t_n$ \;
        }
        \If{$t = t_{k}$}{
          \If{$s = 0$}{
            $s \longleftarrow 1$ \;
          } \Else {
            $s \longleftarrow 0$ \;
          }
        }
        \If{$t > t_{k}$}{
          $t_{\text{prev}} \longleftarrow$ time of first transition after $t_n$ \;
        }

        \If{$\theta(t) \leq \theta_{\text{sp}} - \frac{\Delta}{2} + \alpha$} {
          $s_i(t) \longleftarrow 0$ \;
        }
        \ElseIf{$\theta(t) \geq \theta_{\text{sp}} + \frac{\Delta}{2} - \alpha$} {
          $s_i(t) \longleftarrow 1$ \;
        }
        \If{$s = 1$}{
          Turn cooling power on \;
        }      
        $\alpha \longleftarrow \alpha e^{-a\epsilon}$ \; 
        $t \longleftarrow t + \epsilon$ \;
      }
      $t_{k+1} \longleftarrow (t_{\text{prev}} + t_{\text{next}})/2$ \;
      $k \longleftarrow k + 1$ \;
      $t \longleftarrow 0$ \;
    }
  }
  \caption{Desynchronize Cooling TCL}
\end{algorithm}

\section{Simulation examples}
\label{sec:example}
Armed with Algorithm 1, we re-visit the simulation examples in Section \ref{sec:problem}. Figure \ref{fig:ex1} shows the temperature curves of a subset of a population of 10,000 identical TCLs, simulated under the same conditions as above, i.e., $R = 2 ^\circ$C/kW, $C = 5$ kWh$/^\circ$C, $P = 14$ kW, $\eta = 2.5$, $\theta_{\text{sp}} = 20^\circ$C, $\Delta = 1^\circ$C and $\theta_{\infty} = 28^\circ$C. The top subplot shows the behavior without desynchronization, while the bottom plot shows the behavior with Algorithm 1 applied. The temperature curves clearly look more `jumbled' after the step, which indicates a greater degree of desynchronization, and the operation intervals can be seen to return to their original size after in the space of a few hours.   

\begin{figure}
\begin{center}
\includegraphics[width=0.8\linewidth]{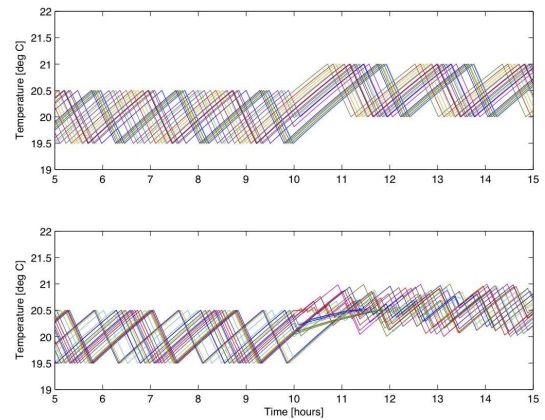}
\caption{Simulation with 10,000 identical TCLs subjected to a step change at time $t=10$ h. Top: no desynchronization; bottom: with synchronization}
\label{fig:ex1} 
\end{center}
\end{figure}

Figure \ref{fig:ex2} shows the corresponding power curves. The effect of the desynchronization is very obvious almost immediately after the step, as the amplitude of already the first peak is less than the case without desynchronization, and most of the oscillations are suppressed after about one period. We also notice that there are some oscillations remaining, which will gradually die out due to the long-term behavior of the adaptive algorithm.   
\begin{figure}
\begin{center}
\includegraphics[width=0.7\linewidth]{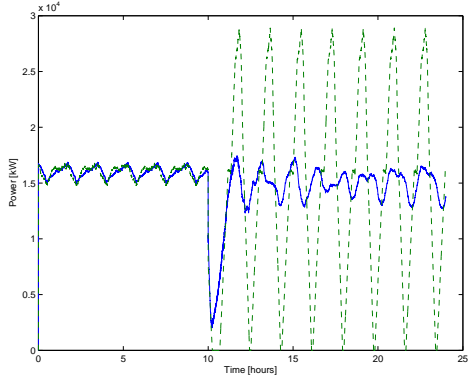}
\caption{Simulation with 10,000 identical TCLs subjected to a step change at time $t=10$ h. Full line: with desynchronization; dashed: without synchronization}
\label{fig:ex2} 
\end{center}
\end{figure}

Figure \ref{fig:ex3} shows a simulation with a heterogenous population of TCLs. Here, as opposed to Figure \ref{fig:ex2}, the parasitic oscillations die out in the case without desynchronization as well. However, Algorithm 1 clearly speeds up the process, however. Note that the oscillations that can be seen before the setpoint change is due to the initialization of the TCL on/off states, which does not match exactly with the steady-state distribution (the desynchronization is not active before $t = 10$ h). Note also that, by changing the setpoint of all the TCLs in the simulation by half a degree, it was possible to reduce the mean total power consumption by approximately 2 MW after the step.
\begin{figure}
\begin{center}
\includegraphics[width=0.7\linewidth]{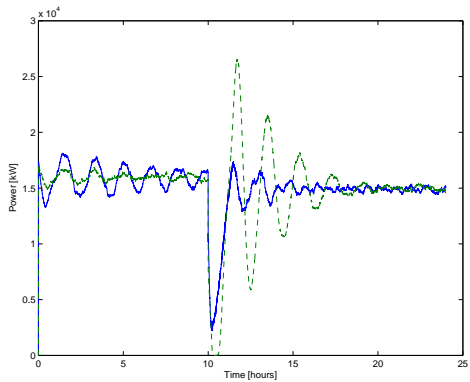}
\caption{Simulation with 10,000 inhomogeneous TCLs subjected to a step change at time $t = 10$ h.  Full line: with desynchronization; dashed: without synchronization}
\label{fig:ex3} 
\end{center}
\end{figure}

\section{Discussion} \label{sec:disc}
This paper presented a novel algorithm for counter-acting unwanted oscillations caused by synchronization of populations of temperature controlled loads, which requires neither central management of the individual units nor communication between units. The algorithm comprises two main part, a `fast' randomization of the temperature bands within which the TCLs operate, and a `slow' adaptive adjustment of enforced switch timings, which maximizes desynchronization over time. We presented a formal proof of convergence of homogeneous populations to the desynchronized status, as well as simulations that indicate that the algorithm is able to effectively dampen power consumption oscillations for both homogeneous and heterogeneous populations. 

However, the simulations also indicate that even with active desynchronization, it is hard to avoid large peaks right after a common broadcast setpoint step. Thus, high-level model-based control along the lines of the work presented in \cite{Kundu:2011}, will probably be preferable to simple steps.

\bibliographystyle{unsrt}


\end{document}